\documentclass[11pt,twoside]{article}

\usepackage{mathrsfs,amsfonts,amsmath,amssymb}
\usepackage{latexsym}
\usepackage{comment}

\newtheorem{satz}{Theorem}
\newtheorem{proposition}[satz]{Proposition}
\newtheorem{theorem}[satz]{Theorem}
\newtheorem{lemma}[satz]{Lemma}
\newtheorem{definition}[satz]{Definition}
\newtheorem{corollary}[satz]{Corollary}
\newtheorem{remark}[satz]{Remark}
\newtheorem{example}[satz]{Example}

\def\N{\mathbb {N}}
\def\Z{\mathbb {Z}}
\def\F{\mathbb {F}}
\def\E{\mathsf{E}}

\def\a{\alpha}

\def\C{\mathbb{C}}

\def\d{\delta}
\def\o{\omega}
\def\({\big (}
\def\){\big )}

\def\codim{{\rm codim}}
\def\le{\leqslant}
\def\ge{\geqslant}
\def\_phi{\varphi}
\def\eps{\varepsilon}

\def\Gr{{\mathbf G}}
\def\FF{\widehat}

\def\ov{\overline}
\def\Spec{{\rm Spec\,}}

\def\Cf{{\mathcal C}}
\def\mP{{\mathcal P}}
\def\mD{{\mathcal D}}

\def\C{\mathbb{C}}
\def\R{\mathbb{R}}

\def\L{\mathcal{L}}

\newcommand{\bp}{\bigskip}

\author{I.D. Shkredov}
\title{Some new results on  the higher energies I}
\date{}
\begin{document}
	\maketitle


\begin{center}
	Annotation.
\end{center}

{\it \small
    We obtain a generalization of the recent Kelley--Meka result on sets avoiding arithmetic progressions of length three. 
    In our proof we develop the theory of the higher energies. 
    Also, we discuss the case of longer arithmetic progressions, as well as a general family of norms, which includes the higher energies norms and Gowers norms. 
}
\\

\section{Introduction}

The famous  Erd\H{o}s--Tur\'an conjecture \cite{ET} asks is it true that for any integer $k\ge 3$  any set of positive integers $A=\{n_1 < n_2 <\dots < n_m < \dots \}$
satisfying 
\begin{equation}\label{conj:ET}
    \sum_{j=1}^\infty \frac{1}{n_j} = \infty 
\end{equation} 
contains an arithmetic progression of length $k$ (we say $A$ has APk), that is the sequence of the form $x,x+y,\dots,x+(k-1)y \in A$? 
This question has a rich history  see, e.g., \cite{Gowers_4}, \cite{Gowers_m} or \cite{sh_Sz_survey}  and is considered a central one in the area of classical additive combinatorics due to its connection with many adjecent fields as combinatorical ergodic theory and graphs/hypergraphs theory, we just mention some papers 
\cite{Szemeredi_4}, \cite{Szemeredi_m}, \cite{furstenberg2014recurrence}, \cite{furstenberg1982ergodic}, \cite{shelah1988primitive}, \cite{gowers2007hypergraph}, \cite{van1927beweis}, \cite{Tao_removal}, \cite{RS_regularity}, \cite{NRS_counting} etc. 
If one defines 
\[
    r_k (N) = \frac{1}{N} \max\{ |A| ~:~ A\subseteq \{1,\dots, N\}\,, \quad A \mbox{ has no APk} \} \,,
\]
then the condition \eqref{conj:ET} means, roughly, that 
\begin{equation}\label{conj:ET_rk}
    r_k (N) \ll \frac{1}{\log N \cdot (\log \log N)^{1+\eps}} \,, \quad \quad N \to \infty 
\end{equation} 
for an arbitrary $\eps>0$.

The case of arithmetic progressions of length   three was considered to be special thanks to the Fourier approach of Roth \cite{Roth1953}, the required information and references can be found in \cite{bloom2023kelley}, \cite{kelley2023strong}, as well as in \cite{sh_Sz_survey}. 
Bloom and Sisask in \cite{bloom2020breaking} proved that $r_3 (N) \ll (\log N)^{-1-c_1}$ for a certain $c_1>0$ and hence established  conjecture \eqref{conj:ET_rk} in the case of $k=3$. 
Recently, Kelley and Meka \cite{kelley2023strong} achieved a remarkable progress in this question and proved that 
\[
    r_k (N) \ll \exp (- O((\log N)^{c_1})) \,,
\]
where $c_1>0$ is an absolute constant. 
One of the ideas of paper  \cite{kelley2023strong} was to use the higher energy $\E^k_2$ and the uniformity relatively to $\E^k_2$ (all definitions can be found in Sections \ref{sec:def}, \ref{sec:E^k_l}) with a growing parameter $k$ to control  the number of arithmetic progressions in an arbitrary set. 
Namely, bound \eqref{conj:ET_rk} is an immediate consequence of   the following result (for simplicity we consider the group $\F_p^n$).

\begin{theorem}
    Let $\Gr=\F_p^n$, $A\subseteq \Gr$ be a set, $|A|=\d N$, and $\eps>0$ be a parameter.
    Then there is a subspace $V\subseteq \Gr$ and $x\in \Gr$ such that $A\cap (V+x)$ is $\eps$--uniform relatively to $\E^k_2$, $\mu_{V+x} (A) \ge \d$, and 
\begin{equation}\label{f:uniform_KM_intr}
    \codim V \ll \eps^{-14} k^4 \L^3 (\d) \L^{2} (\eps \d) \,.
\end{equation}
\label{t:uniform_KM_intr}
\end{theorem}


The aim of  this paper is to generalize Kelley--Meka  results to a more wide additive--combinatorial family of energies $\E^k_l$ see, e.g., \cite{sh_Ek}. In our regime the parameter $l$ below is $l=O(1)$.

\begin{theorem}
    Let $\Gr=\F_p^n$, $A\subseteq \Gr$ be a set, $|A|=\d N$, and $\eps \in (0,1]$ be a parameter.
    Then there is a subspace $V\subseteq \Gr$ and $x\in \Gr$ such that $A\cap (V+x)$ is $\eps$--uniform relatively to $\E^k_l$, $\mu_{V+x}(A) \ge \d$ and 
\begin{equation}\label{f:uniform_KM_k_intr}
    \codim V \ll 
\eps^{-28l^l} (8l)^{28l^l} k^4 \L^{4l} (\eps) \L^{5l} (\d)
    \,.
\end{equation}
\label{t:uniform_KM_k_intr}
\end{theorem}

Theorem \ref{t:uniform_KM_k_intr} is interesting in its own right and can be used to 
solve more general equations and systems than $x+y=2z$, which corresponds to the case of AP3. 
Our forthcoming paper will be devoted to finding some applications of this result.

Usually the case of the groups $\F_p^n$ for prime $p$ is considered to be a model one, see the excellent survey \cite{Green_models} about this theme and an appropriate generalization for all abelian groups $\Gr$ is known to be a technical task. We will prove the correspondent result in our forthcoming paper.

The approach develops the strategy of \cite{Gowers_4}, \cite{Gowers_m}, the method of the higher energies (see, e.g., \cite{SS_higher}, \cite{sh_Ek}) and of course \cite{kelley2023strong}. 
Also, we extensively use a brilliant exposition \cite{bloom2023kelley}, where the Kelley--Meka results were discussed.

In the appendix we discuss the original Erd\H{o}s--Tur\'an conjecture, that is the case of longer arithmetic progressions and show that there is series of difficulties on the  conceptual and on the technical levels which make the question about generalizations of  the methods from \cite{kelley2023strong} rather hard. 
The author thinks this part is also interesting in its own right because it allows us to understand the limits of the  Kelley--Meka approach.  
Also, we consider a general family of norms, which includes, simultaneously,  the norms $\E^k_l$ above, as well as the classical Gowers norms \cite{Gowers_m}.

\section{Definitions and preliminaries}  
\label{sec:def}

Let  $\Gr$ be a finite abelian group 
and
denote by $N$ the cardinality of $\Gr$.
We use the same capital letter to denote  set $A\subseteq \Gr$ and   its characteristic function $A: \Gr \to \{0,1 \}$. 
Let us define $\mu_A (x) = A(x)/|A|$, that is $\sum_{x\in \Gr} \mu_A (x) =1$. 
Finally, let $f_A(x) = A(x) - |A|/N$ be the {\it balanced function} of $A$.
Given two sets $A,B\subset \Gr$, define  
the {\it sumset} 
of $A$ and $B$ as 
$$A+B:=\{a+b ~:~ a\in{A},\,b\in{B}\}\,.$$
In a similar way we define the {\it difference sets} and {\it higher sumsets}, e.g., $2A-A$ is $A+A-A$.

Let $f$ be a function from $\Gr$ to $\mathbb{C}.$  We denote the Fourier transform of $f$ by~$\FF{f},$
\begin{equation}\label{F:Fourier}
  \FF{f}(\xi) =  \sum_{x \in \Gr} f(x) \overline{\chi (x)} \,,
\end{equation}
where $\chi \in \FF{\Gr}$ is a character of $\Gr$. 
We rely on the following basic identities
\begin{equation}\label{F_Par}
    \sum_{x\in \Gr} |f(x)|^2
        =
            \frac{1}{N} \sum_{\chi \in \FF{\Gr}} \big|\widehat{f} (\chi)\big|^2 \,,
\end{equation}
and
\begin{equation}\label{f:inverse}
    f(x) = \frac{1}{N} \sum_{\chi \in \FF{\Gr}} \FF{f}(\chi)  \chi(x) \,.
\end{equation}
If
$$
    (f*g) (x) := \sum_{y\in \Gr} f(y) g(x-y) \quad \mbox{ and } \quad (f\circ g) (x) := \sum_{y\in \Gr} f(y) g(y+x) \,,
$$
 then
\begin{equation}\label{f:F_svertka}
    \FF{f*g} = \FF{f} \FF{g} 
\end{equation}
and similar for $f\circ g$. 
 Clearly,  $(f*g) (x) = (g*f) (x)$ and $(f\circ g)(x) = (g \circ f) (-x)$, $x\in \Gr$.
 The $k$--fold convolution, $k\in \N$  we denote by 
 $f^{(k)}$,
 so $f^{(2)} = f*f$ and $f^{(3)} = f*f*f$ for example.

We need some formalism concerning higher convolutions see, e.g.,  
\cite{sh_Ek}. 
Let  $l$ be a positive integer. 
Consider 
two operators $\mD_l$, $\mP_l : \Gr \to \Gr^l$ such that for a variable $x$ one has $\mD_l (x) = (x,\dots,x) \in \Gr^l$ and $\mP_l (x) = (x_1,\dots,x_l) \in \Gr^l$. 
Notice that $\mP_1 (x) = \mD_1 (x) = x$. 
In the same way these operators act on functions, e.g., 
$\mP_l (f) (x_1,\dots,x_l) = f(x_1) \dots f(x_l)$ for $f:\Gr \to \C$ and for $F:\Gr^l \to \C$ one has $\mD_l (F)(x_1,\dots,x_l) =F(x_1,\dots,x_1)$ if $x_1=\dots=x_l$ and zero otherwise. 
Now given  a function $f:\Gr \to \C$ and a positive integer $l$ define the generalized convolution 
\begin{equation}\label{def:Cf}
    \Cf_l (f) (x_1,\dots,x_l) = \sum_{z\in \Gr} f(z+x_1) \dots f(z+x_l) = (\mD_l (\Gr) \circ \mP_l (f)) (x_1,\dots,x_l) 
\end{equation} 
\begin{equation}\label{def:Cf+}
    := \sum_{z\in \Gr} f_{x_1,\dots,x_l} (z) \,.
\end{equation} 
In a similar way we can consider $\Cf_l (f_1,\dots,f_l) (x_1,\dots,x_l)$ for any functions $f_1,\dots,f_l : \Gr \to \C$. 
One has 
\begin{equation}\label{f:E^k_l_symmetries} 
    \Cf_l (f) (x_1,\dots,x_l) = \Cf_l (f) (x_1+w,\dots,x_l+w) = \Cf_l (f) ((x_1,\dots,x_l) + \mD_l (w))
\end{equation} 
for any $w\in \Gr$. 
Let us 
emphasise 
that definitions \eqref{def:Cf}, \eqref{def:Cf+} differ slightly from the usual one, see, e.g., \cite{sh_Ek} by a linear change of the variables. Namely, it is a little bit more traditional to put 
\begin{equation}\label{f:E^k_l:C,C'-}
    f'_{x_1,\dots,x_{l}} (z) = f_{0,x_1,\dots,x_l} (z)= f(z) f(z+x_1) \dots f(z+x_{l}) \,,
\end{equation} 
and
\begin{equation}\label{f:E^k_l:C,C'}
    \Cf'_{l+1} (f) (x_1,\dots,x_l) = \sum_{z\in \Gr} f(z) f(z+x_1) \dots f(z+x_l) =  \Cf_l (f) (0, x_1,\dots,x_l)\,.
\end{equation} 
Definitions \eqref{f:E^k_l:C,C'-}, \eqref{f:E^k_l:C,C'} have  an advantage that they allow to consider infinite groups $\Gr$ as well. 
To this end we use the dual notation  $\| f\|^{kl}_{\E^k_l} = \bar{\E}^{k}_l (f) = N^{-1} \E^k_l (f)$. 
Now having $k,l\ge 2$ and a function $f:\Gr \to \C$ one can consider 
\begin{equation}\label{def:E_kl}
    \E^k_{l} (f) = \sum_{x_1,\dots,x_l} \Cf^k_l (f) (x_1,\dots,x_l) = \E^l_{k} (f) 
\end{equation}
and it was showed in 
\cite[Proposition 30]{sh_Ek} that for a real function $f$ and even $k,l$ the formula $(\E^k_{l} (f))^{1/kl}$ defines a norm of our function $f$. 
If one put $l=1$ in \eqref{def:E_kl}, then we formally obtain $\E_1^k (f) =N(\sum_z f(x) )^k$ and this is not a norm for any $k$. Nevertheless, it is convenient to consider the quantities $\E_1^k (f)$ sometimes. 
 Notice that $\E^k_{l} (f) \ge 0$, 
 provided  at least one of $k,l$ is even but, nevertheless,  it cannot be a norm in this case, see \cite[Sections 4,7]{sh_Ek}.  
 A general family of norms, which includes the norms above is considered in the second part of the appendix. In particular, one can find the discussed properties of the energies $\E_{l}^k (f)$. 


Let $\eps \in (0,1]$ be a real number. We write $\mathcal{L} (\eps)$ for $\log(2/\eps)$.
Let us 
make a convention 
that if a product is taken over an empty set, then it equals one. 
The signs $\ll$ and $\gg$ are the usual Vinogradov symbols.
When the constants in the signs  depend on a parameter $M$, we write $\ll_M$ and $\gg_M$. 
All logarithms are to base $2$.
By $\F_p$ denote $\F_p = \Z/p\Z$ for a prime $p$. 
Let us denote by $[n]$ the set $\{1,2,\dots, n\}$.


\section{Some results on $\E^k_l$--norms}
\label{sec:E^k_l}

In this section we obtain some generalization of Kelley--Meka results which were obtained for $\E_2^k$--norm to  $\E_l^k$--norm. Also, we  discuss some special properties of such norms. Our results naturally break down into two cases: uniform and non--uniform.

\subsection{Uniform sets in the sense of $\E^k_l$--norm}
\label{subsec:uniformity}

Let us give the  main definition of this subsection. 

\begin{definition}
    Let $\Gr$ be an abelian group, $A\subseteq \Gr$ be a set,  $|A|=\d N$, and $\eps>0$ be a parameter.
    Then we say that $A$ is $\eps$--uniform relatively to (the energy) $\E^k_l$ if 
\begin{equation}\label{def:uniformity_kl} 
    \| f_A\|^{kl}_{\E^k_l} \le \eps^{kl} \d^{kl} N^{k+l} \,.
\end{equation}
\end{definition}

Usually the number $\eps$ belongs to $(0,1]$ but sometimes $\eps>1$ and hence one can consider the quantity $\eps$ as the  definition of the energy $ \| f_A\|^{kl}_{\E^k_l}$, that is 
$\| f_A\|^{kl}_{\E^k_l} := \eps^{kl} \d^{kl} N^{k+l}$. 
Further by the H\"older inequality, we have 
\begin{equation}\label{f:Holder_Ekl}
    (\E^{k-1}_{l} (f))^{k} \le (\E^{k}_{l} (f))^{k-1} N^l 
\end{equation} 
and hence if $A$ is $\eps$--uniform relatively to $\E^k_l$, then $A$ is $\eps$--uniform relatively to $\E^{k'}_{l'}$ for $k'\le k$, $l'\le l$. 
On the other hand, 
it is easy to see that a smaller norm does not control the higher one. 


\begin{example}
Let $\Gr = \F_2^n$, $H< \F_2^n$, $\Lambda \subseteq \F_2^n/H$ be a random set such that $|\Lambda| = \d N/|H|$. 
Also, suppose that $\d^2 \gg |H|/N$ and thus with high probability $\Lambda - \Lambda \approx \F_2^n/H$. 
Let $A$ be the direct sum of $H$ and $\Lambda$, then $|A|=\d N$. 
It is easy to see that for a random $x\in A-A \approx \Gr$ one has $|A_x| \sim \d^2 N$ but for $x,y\in H$ one has $A_x= A$ and $A_{x,y} = A$.
Thus for any $k\ge 2$ the following holds 
\[\E_2^k (A) \sim  (\d^2 N)^k N + (\d N)^k |H| \sim (\d^2 N)^k N \,, \]
provided $|H| \ll \d^k N$ but taking an arbitrary $l$,  we see that
\[\E^l_3 (A) \sim (\d^3 N)^l N^2 + (\d N)^l |H|^2 \gg (\d N)^l |H|^2 \,, 
\]
provided $|H| \gg \d^l N$. 
It follows that one can take $l=k+1$ and construct a set $A$ such that $A$ is  $\E_2^k$--uniform but not $\E_3^{k+1}$--uniform. 
\end{example}



Now let us obtain the characteristic property of the energy $\E^k_l$.

\begin{lemma}
    Let $l,k \ge 2$ be even numbers, 
    and $A_j \subseteq \Gr$, $j\in [l]$ be sets. 
    Then for any function $g:\Gr \to \R$ one has 
\begin{equation}\label{f:uniformity} 
    \sum_{x} \prod_{j=1}^l (f_{A_j} \circ g) (x)
        \le 
            \| g\|^{l(1-1/k)}_1 \|\Cf_l (g)\|^{1/k}_\infty  \cdot \prod_{j=1}^l \| f_{A_j} \|_{\ov{\E}^k_l} \,.
\end{equation}
    If the sets $A_j$ are the same, then the same is true for even $k$ only. 
\label{l:uniformity} 
\end{lemma} 
\begin{proof} 
    By the H\"older inequality and \cite[Lemma 29]{sh_Ek} and the duality one has 
\[
    \sum_{x} \prod_{j=1}^l (f_{A_j} \circ g) (x)  = N^{-1} \sum_{|z|=l} \Cf_l (f_{A_1}, \dots, f_{A_l}) (z) \Cf_l (g) (z) 
\]
\[
    \le 
    N^{-1+1/k} \prod_{j=1}^l \| f_{A_j} \|_{\ov{\E}^k_l} \cdot 
    \left( \sum_{|z|=l} |\Cf^{k/(k-1)}_l (g) (z)| \right)^{1-1/k}
    \le 
    \prod_{j=1}^l \| f_{A_j} \|_{\ov{\E}^k_l} \cdot 
    \|\Cf_l (g) \|^{1/k}_\infty \| g \|^{l(1-1/k)}_1
\]
This completes the proof. 
$\hfill\Box$
\end{proof}

\begin{corollary}
    Let $A,B\subseteq \Gr$ be sets, $|A|=\d N$, $|B| = \beta N$, and $l$ be a positive integer.  
    Take $k =2 \lceil e l \log (1/\beta) \rceil$
    and suppose that $A$ is 
    $\eps$--uniform relatively to the energy $\E^k_l$.
    Then
\begin{equation}\label{f:A_circ_B}
    \sum_{x} (A\circ B)^l (x) \le  \delta^l |B|^l N  \cdot \min\{ 1.25 (1+\eps)^l, (1+1.25 \eps)^l\}  \,.
\end{equation}
\label{c:A_circ_B}
\end{corollary} 
\begin{proof} 
    Using the formula $A(x) = f_A(x) +\d$, combining with Lemma \ref{l:uniformity},  we see that the left--hand side of \eqref{f:A_circ_B} is 
\[
    \sum_{j=0}^l \binom{l}{j} (\d |B|)^{l-j} \sum_x (f\circ B)^j (x) 
    \le 
    \sum_{j=0}^l \binom{l}{j} (\d |B|)^{l-j} \| f\|^j_{\ov{\E}^k_j} |B|^{j(1-1/k)+1/k} 
\]
\[
    \le \d^l |B|^{l} N \sum_{j=0}^l \binom{l}{j} \eps^j \beta^{-(j-1)/k}
    \le \frac{5}{4} \delta^l |B|^l N (1+\eps)^l 
\]
as required and similarly the second bound. 
This completes the proof. 
$\hfill\Box$
\end{proof}

\bp

   Let us remark that, of course,  the energy $\E_2^k$ solely allows us to control sums from \eqref{f:A_circ_B} but our task is to obtain the correct power of $\d$ and $|B|$ in the right--hand side of this estimate. 


\bp 

We need one more result about uniform set, which is useful for applications.

\begin{lemma}
    Let $k,l$ be even numbers,
    $A_1,\dots, A_l \subseteq \Gr$ be sets, $|A_j| = \d_j N$. 
    Suppose that all $A_j$ are $\eps$--uniform relatively to $\E^k_l$ and 
    $2 l \eps^k \le 1$. 
    Then 
\begin{equation}\label{f:dispersion}
    \sum_{|x|=l} \left( \Cf_l (A_1,\dots, A_l) (x) - N\prod_{j=1}^l \d_j \right)^k
    \le 
    \eps^{k} 
    l 2^{kl+1}
    N^{l+k} \left( \prod_{j=1}^l \d_j  \right)^k \,.
\end{equation} 
\label{l:dispersion}
\end{lemma}
\begin{proof} 
    Put $\Pi = \prod_{j=1}^l \d_j$. 
Then the left--hand side of \eqref{f:dispersion} is 
\[
    \sigma:= \sum_{|x|=l} \left( \sum_{\emptyset \neq S \subseteq [l]} \Cf_{l} (f_{1}, \dots, f_l) (x) \right)^k
    = \sum_{|x|=l} \left( \sum_{\emptyset \neq S \subseteq [l]} F_S (x) \right)^k \,,
\]
    where for $j\in S$ we put $f_j=f_{A_j}$ and if $j\notin S$, then $f_j = \d_j$. 
    Using $\eps$--uniformity of all sets $A_j$, combining with \cite[Lemma 29]{sh_Ek}, we get 
\[
    \| F_S\|^{kl}_{\E^k_l} \le \eps^{|S|k} \Pi^k N^{l+k} \,.
\]
    Thus by the H\"older inequality one has
\[
    \sigma  \le 2^{kl} \sum_{\emptyset \neq S \subseteq [l]} \| F_S\|^{kl}_{\E^k_l}
    \le 
    2^{kl} \Pi^k N^{k+l} ((1+\eps^k)^l - 1)
    \le l2^{kl+1} \eps^k \Pi^k N^{k+l} 
\]
as required. 
$\hfill\Box$
\end{proof}

\bp

Let us consider one more  example which shows that one can delete/add a tiny subset from a non--uniform set to obtain a uniform one. This phenomenon has no place if we consider the classical uniformity in terms of the Fourier transform or in terms of Gowers norms \cite{Gowers_m}, say. The reason is  normalization \eqref{def:uniformity_kl}, of course. 


\begin{example} 
\label{exm:removing}
Let $\Gr = \F_2^n$, $H< \F_2^n$, $|H|=\beta N$, $\Lambda \subseteq \Gr$ be a random set, $|\Lambda| = \d N$, $\beta \le \d$ and put $A=\tilde{H} \bigsqcup \Lambda$, where $\tilde{H} = H \setminus \Lambda$. Then with high probability  $|\tilde{H}| \sim \beta(1-\d) N = \eps |A|$, the set $\Lambda$ is uniform in any possible sense but $A$ is non $\eta$--uniform set with rather large $\eta$. 
Indeed, by Kelley--Meka method \cite{kelley2023strong} or just see Lemma \ref{l:uniformity}, we know that 
\begin{equation}\label{tmp:23.03_1}
    \sigma:= \sum_{x} (A\circ A)(x) H(x) = (\d+(1-\d)\beta)^2 \beta N^2 + \theta \eta^2 \d^2 \beta N^2 \,,
\end{equation}
    where $|\theta| \le 4$, say,  is a certain number and  
    $A$ is supposed to be $\E_2^k$--uniform with $k \sim \L(\beta)$.
    On the other hand, the direct calculation shows 
\begin{equation}\label{tmp:23.03_2}
    \sigma =  |\tilde{H}|^2 + 2\d\beta^2 (1-\d) N^2 + \d^2 |H|N 
\end{equation}
    plus a negligible  error term.
    Comparing \eqref{tmp:23.03_1} and \eqref{tmp:23.03_2}, we obtain 
\[
    \eta^2 \d^3 \eps \gg 2\d \beta \eps + \beta \eps^2 -  \eps^2 \d^2 \gg \d \eps^3 
\]
    and thus $\eta \gg \eps/\d$ which is much larger than $\eps$ for small $\d$.

    Similarly, 
    one can 
    show that deleting 
    a subspace $H$ from a random set $\Lambda$, $|H| = \eps |A|$, $H$ lives on the first coordinates, say, we obtain a non $\eta$--uniform set with $\eta \gg 1$ thanks to the equality $\sum_{x} H(x) (A\circ A) (x) = 0$. 
\end{example} 


\subsection{Non--uniformity and almost periodicity}
\label{subsec:NU+AP}

The aim of this subsection is to obtain Sanders' almost periodicity result for {\it higher  convolutions}, see Lemma \ref{l:Sanders_V} below. 


At the beginning 
we want to transfer 
a lower bound for the 
energy $\E^k_l (f_A)$ to 
the largeness of the energy 
$\E^k_l (A)$.  
We follow a more simple method from \cite{bloom2023kelley} which differs from the approach of \cite{kelley2023strong} by some logarithms. The dependence on $l$ in the first multiple in  \eqref{f:e_to_d} is, probably, can be improved  significantly (also, see Remark \ref{r:ET_fail} from the appendix)
but in our regime $l=O(1)$ and 
thus it is not so critical for us.

\begin{lemma}
    Let $A\subseteq \Gr$ be a set, $|A|= \d N$ and 
    $\eps>0$ be a parameter, $\eps_* := \min \{\eps, 1\}$.
    Suppose that for an odd $k\ge 5$ one has 
\begin{equation}\label{cond:e_to_d}
    \E_l^k (f_A) = \eps^{lk} \d^{lk} N^{l+k} \,,
\end{equation} 
    and that for $k_* = O(kl \eps^{-l}_* \L(\eps_*))$ the set $A$ is $\frac{\eps \eps^{l-1}_*}{8l}$--uniform relatively to $\E^{k_*}_{l-1}$.
    Then there is an even  $k_1 \le k_*$ such that 
\begin{equation}\label{f:e_to_d}
    \E_l^{k_1} (A) \ge \left(1+\frac{\eps \eps^{l-1}_*}{8l}  \right)^{lk_1} \d^{lk_1} N^{l+k_1} \,.
\end{equation} 
\label{l:e_to_d}
\end{lemma}
\begin{proof} 
    Write $f(x)= f_A (x)$. 
   We have 
\[
    \eps^{lk} \d^{lk} N^{l+k} \le \sum_{|x|=l} \Cf^{k-1}_l (f) (x) |\Cf_l (f) (x)| + \sum_{|x|=l} \Cf^{k}_l (f) (x) 
    = 2 \sum_{|x|=l} \Cf^{k-1}_l (f) (x) \cdot \max\{ 0, \Cf_l (f) (x)\} \,,
\]
    and thus considering the set $P = \{ x ~:~ \Cf_l (f) (x) \ge 0\}$, we get 
\begin{equation}\label{tmp:03.03_1}
    \sum_{x\in P} \Cf^{k}_l (f) (x) \ge 2^{-1} \eps^{lk} \d^{lk} N^{l+k} \,.
\end{equation} 
    Now let consider the subset of the set $P$
\[
    P_\eps := \{ x ~:~ \Cf^{}_l (f) (x) \ge \frac{3}{4} \eps^l \d^l N\} \,.
\]
    Then we have 
\begin{equation}\label{tmp:03.03_2}
    \sum_{x\notin P_\eps} \Cf^{k}_l (f) (x) \le 
    \left( \frac{3}{4} \eps^l \d^l N \right)^k N^l 
    \le 
    2^{-2} \eps^{lk} \d^{lk} N^{l+k} \,.  
\end{equation} 
    Combining \eqref{tmp:03.03_1}, \eqref{tmp:03.03_2} and using the H\"older inequality, we obtain 
\begin{equation}\label{tmp:03.03_3}
    |P_\eps| \E^{2k}_{l} (f) \ge 2^{-4} \eps^{2lk} \d^{2lk} N^{2l+2k} \,. 
\end{equation} 
    By the norm property of $\E^{2k}_{l} (f)$ for positive functions, see \cite[Proposition 30]{sh_Ek} one has 
\[
    \E^{2k}_{l} (f) \le \left( \| A\|_{\E^{2k}_l} + \| \d\|_{\E^{2k}_l} \right)^{2kl} 
    \le 
    (2+\eps/2)^{2kl} \d^{2kl} N^{l+2k} 
\]
    otherwise there is nothing to prove with $k_1=2k$ and much larger $\eps$. 
    Thus we derive from \eqref{tmp:03.03_3} that 
    $|P_\eps| \ge (2\eps_*/5)^{2kl} N^l$. 
    Now 
\[
    \Cf_l (A) (x) = \Cf_l (f+\d) (x) =\d^l N + \Cf_l (f) + \sum_{ S\subseteq [l] ~:~ 1\le |S|<l} \d^{l-|S|} \Cf_{|S|} (f) (x_S) = \d^l N + \Cf_l (f) + \mathcal{E} (x) \,,
\]
    where for a set $S\subseteq [l]$
    the vector $x_S$ has coordinates $x_j$, $j\in S$. 
    By the triangle inequality for $L_{k_1}$--norm, we have 
\begin{equation}\label{tmp:06.03_1}
    (\E^{k_1}_{l} (A))^{1/k_1} = \| \Cf_l (A) \|_{k_1} = \| \Cf_l (f+\d) \|_{k_1}
    \ge 
    \| \d^l N + \Cf_l (f) \|_{k_1} - \| \mathcal{E}\|_{k_1} \,.
\end{equation} 
    Using our bound for the cardinality of the set $P_\eps$, we get
\[
    \| \d^l N + \Cf_l (f) \|^{k_1}_{k_1} \ge \sum_{x\in P_\eps} (\d^l N + \Cf_l (f))^{k_1} (x) \ge (2\eps_*/5)^{2kl} N^l \cdot (1+3\eps^l/4)^{k_1} \d^{lk_1} N^{k_1} 
\]
\begin{equation}\label{tmp:06.03_2}
    \ge 
    (1+\eps^l/2)^{k_1} \d^{lk_1} N^{k_1+l}  \,,
\end{equation} 
    provided $k_1 \ge 20 kl \eps^{-l}_* \L(\eps_*)$. 
    On the other hand, by our assumption the set $A$ is $\zeta:=\frac{\eps^l}{8l}$--uniform relatively to 
    $\E^{k_*}_j$ for all $j<l$ and $k_* = k_1$.
    It follows that 
\begin{equation}\label{tmp:06.03_3}
    \| \mathcal{E}\|_{k_1} \le \sum_{ S\subseteq [l] ~:~ 1\le |S|<l} \d^{l-|S|}
    N^{\frac{l-|S|}{k_1}} \| f\|^{|S|}_{\E^{k_1}_{|S|}} 
    \le \d^l 
    N^{\frac{l+k_1}{k_1}} ((1+\zeta)^{l}-1) \,.
\end{equation} 
    Combining \eqref{tmp:06.03_1}, \eqref{tmp:06.03_2} and \eqref{tmp:06.03_3}, we obtain 
\[
    \E^{k_1}_{l} (A) \ge \d^{lk_1} N^{k_1+l} \left( 2 +\eps^l/2  -  (1+\zeta)^{l} ) \right)^{k_1}
    \ge  \d^{lk_1} N^{k_1+l} \left(1+\frac{\eps \eps^{l-1}_*}{8l} \right)^{lk_1} 
\]
as required. 
$\hfill\Box$
\end{proof}

\bp 

Now we use duality \eqref{def:E_kl} to  obtain an appropriate version of multi--dimensional version of the Balog--Szemer\'edi--Gowers theorem as was done in \cite{Schoen_BSzG} (also, see \cite[Theorem 17]{sh_str_survey}).
Thanks to duality \eqref{f:E^k_st_duality} one can show that 
a similar 
result takes place for more general energies $\mathcal{E}^k_{s,t}$, see the appendix. 
Of course, in this case one needs to replace 
$\Cf_{|x|} (\cdot) (x)$ to $\Cf_{|x||z|} (\cdot) (x\oplus z)$ or  $\Cf_{|y||z|} (\cdot) (y\oplus z)$ 
and use symmetries \eqref{f:E^k_st_symmetries}  instead of the symmetry \eqref{f:E^k_l_symmetries} below.

\begin{lemma}
    Let $A \subseteq \Gr$ be a set, $|A|= \d N$ and $\eps>0$, $\eta \in (0,1/2)$ be parameters. 
    Suppose that for some integers $k,l\ge 2$ with $kl \ge 4\eps^{-1}_* \L(\eta)$ one has  
\begin{equation}\label{cond:BSzG_k}
    \E^k_l (A) \ge (1+\eps)^{lk} \d^{lk} N^{l+k} \,.
\end{equation} 
    Define the set 
\begin{equation}\label{def:BSzG_k}
    S = \{ |x|=l ~:~ \Cf_l (A) (x) \ge (1+\eps/4)^l \d^l N \} \,.
\end{equation} 
    Then there is a set $B$ such that 
\begin{equation}\label{f:BSzG_k}
    N^{-1} \sum_{|x|=l} S(x) \Cf_l (B) (x) \ge (1-2\eta) |B|^l \,,
\end{equation} 
    and 
$|B| > 2^{-1/(l-1)} (1+\eps)^k  \d^k N$.
\label{l:BSzG_k}
\end{lemma}
\begin{proof} 
    For {\it any} set $S \subseteq \Gr^l$ with the property $S(x + \mD_l(t)) = S(x)$, $t\in \Gr$, $x\in \Gr^l$,  we have 
\[
    \sum_{|x|=l} S(x) \Cf^k_l (A) (x)
    =
    \sum_{|x|=l} S(x) \sum_{|z|=k} A^k (z+\mD_k (x_1)) \dots A^k (z+\mD_k (x_l))
\]
\begin{equation}\label{tmp:28.02_0}
    =
    \sum_{|z|=k}\, \sum_{|x|=l} S(x) A_{z} (x_1) \dots A_z (x_l)
    =
    N^{-1} \sum_{|z|=k}\, \sum_{|x|=l} S(x) \Cf_l (A_z) (x) \,,
\end{equation} 
where we have made the change of the variables $x_j \to x_j+t$ in the last formula. 
Clearly, we have from identity \eqref{f:E^k_l_symmetries} and definition \eqref{def:BSzG_k} that 
    $S(x + \mD_l(t)) = S(x)$ for and $t\in \Gr$ and $x\in \Gr^l$ and thus the argument above can be applied for the set $S$ as well. 
    Thus using the definition of the set $S$, as well as the conditions \eqref{cond:BSzG_k} and  $kl \ge 4\eps^{-1}_* \log (4/\eta)$, we get 
\begin{equation}\label{tmp:28.02_2*}
    \sum_{x\notin S}  \Cf^k_l (A) (x) \le (1+\eps/4)^{lk} \d^{lk} N^{k+l}
    \le 
    2^{-2} \eta (1+\eps)^{lk} \d^{lk} N^{k+l} 
    \le 
    2^{-2} \eta \E^k_l (A) \,.
\end{equation} 
    Now define  the set 
\[
    \Omega = \left\{ |z|=k ~:~ |A_z| \ge 2^{-1/(l-1)} (1+\eps)^k  \d^k N \right\} \,.
\]
    By the definition of the set $\Omega$, we derive 
\begin{equation}\label{tmp:28.02_2**}
\sum_{z\notin \Omega} |A_z|^l \le 
    (\max_{z\notin \Omega} |A_z|)^{l-1} |A|^k N \le 
    2^{-1} (1+\eps)^{lk} \d^{lk} N^{l+k} \le 
    2^{-1} \E^k_l (A)  \,. 
\end{equation} 
In view of 
bounds 
\eqref{tmp:28.02_2*}, \eqref{tmp:28.02_2**} one has 
\begin{equation}\label{f:prob}
    N^{-1} \sum_{z\in \Omega} \left( \sum_{x\in S}  \Cf_l (A_z) (x) - \eta^{-1} \sum_{x\notin S}  \Cf_l (A_z) (x) \right) \ge 2^{-1} \sum_{|z|=k} |A_z|^l - 2^{-1} \E^k_l (A) = 0 \,.
\end{equation} 
    Hence there is $z\in \Omega$ such that inequality \eqref{f:BSzG_k} holds for $B=A_z$ and $|B| > 2^{-1/(l-1)} (1+\eps)^k  \d^k N$ as required. 
$\hfill\Box$
\end{proof}

\bigskip 

Now we need an analogue of the almost periodicity result \cite{CS}  (also, see \cite{sanders2011roth}, \cite{sanders2012bogolyubov}, \cite{sanders2013structure} and, especially, 
\cite[Theorem 3.2]{SS4}) for the higher convolutions. 
This theme is rather well--known and thus we give just a scheme of the proof emphasizing the necessary distinctions we must make.  
For the convolution  $\Cf_{|x||z|} (x\oplus z)$ a similar result takes place, see Lemma  \ref{l:CS_new} from the appendix.

\begin{lemma}
    Let $\Gr=\F_p^n$, $l$ be an integer and $\epsilon  \in (0,1]$ be a real parameter.
    Also, let $B\subseteq \Gr$ be a set, $|B|=\beta N$, and $f:\Gr^l \to [-1,1]$ be a function. 
    Then there is a subspace $V\le \Gr$ with 
\begin{equation}\label{f:Sanders_V_codim}
        \mathrm{codim} V \ll \epsilon^{-2} l \L^2 (\beta) \L^2 (\epsilon \beta^l)
\end{equation} 
    and such that 
\begin{equation}\label{f:Sanders_V}
    \left| \sum_{|x|=l} f(x) (B^l \circ \mD_l (B * \mu_V)) (x) - \sum_{|x|=l} f(x) (B^l \circ \mD_l (B)) (x)  \right| 
    \le \epsilon |B|^{l+1} \,.
\end{equation} 
\label{l:Sanders_V}
\end{lemma} 
\begin{proof}
    We begin with a rather general argument which takes place in any abelian group $\Gr$. 
    Let $k\ge 2$ be an integer parameter and $q\ge 2$ be a real parameter.  
    Applying the Croot--Sisask lemma \cite{CS}, \cite{sanders2012bogolyubov} (clearly, one has  $|\mD_l (B) + \mD_l (\Gr)| \le \beta^{-1} |\mD_l (B)|$, so the set $\mD_l (B)$ has the small doubling), 
    we find a set $T\subseteq \Gr$, $|T|\ge |B| \exp(-O(\epsilon^{-2} qk^2 \log (1/\beta)))$    and such that for any $t\in kT$ the following holds 
\[
    \sum_{|x|=l} \left|  (f \circ \mD_l (B)) (x+\mD_l (t)) - 
    \sum_{|x|=l} (f \circ \mD_l (B)) (x) \right|^q 
\]
\begin{equation}\label{tmp:28.02_5}
        \le 
            \left(\frac{\epsilon}{4} \right)^q \| f\|_q^q |B|^q 
                \le      
                \left(\frac{\epsilon}{4} \right)^q |B|^q N^l \,.
\end{equation}
    Fixing $t\in kT$ and using the H\"older inequality, combining with estimate  \eqref{tmp:28.02_5}, we get 
\[
    \left| \sum_{|x|=l} f(x) (B^l \circ \mD_l (B)) (x+\mD_l (t)) 
    - \sum_{|x|=l} f(x) (B^l \circ \mD_l (B)) (x) \right| 
        \le 
        \frac{\epsilon}{4} |B| N^{l/q} |B|^{l(1-1/q)}
\]
\begin{equation}\label{tmp:28.02_5.5}
    =
        \frac{\epsilon}{4} \beta^{-l/q} |B|^{l+1}
        \le 
        \frac{\epsilon}{2} |B|^{l+1} \,,
\end{equation}
    where we have taken $q=C l \log (1/\beta)$ for a sufficiently large constant $C>0$. 
    It follows that 
\begin{equation}\label{tmp:28.02_6}
    \left| \sum_{|x|=l} f(x) (B^l \circ \mD_l (B * \mu^{(k)}_T)) (x) - \sum_{|x|=l} f(x) (B^l \circ \mD_l (B)) (x) \right| \le 2^{-1} \epsilon |B|^{l+1} \,. 
\end{equation}
    Let us analyze the sum $\sigma:= |\sum_{|x|=l} f(x) (B^l \circ \mD_l (B * \mu^{(k)}_T)) (x)|$ from \eqref{tmp:28.02_6}. 
    Clearly, $\FF{\mu}_T (r_1,\dots, r_l) = |T|^{-1} \FF{T}(r_1+\dots+r_l)$ and thus 
\begin{equation}\label{tmp:28.02_7}
    \sigma \le \frac{|B|}{|T|^k N^l} \sum_{z} |\FF{T} (z)|^k
    \sum_{r_1+\dots+r_l=z} |\FF{f} (r_1,\dots,r_l)| |\FF{B}(r_1)| \dots |\FF{B}(r_l)| \,.
\end{equation}
    As usual let us estimate the last sum over $z\in \Spec_c (T)$, where  $c\in (0,1]$ is a parameter  and  over $z\notin \Spec_c (T)$, 
    where we have defined 
\[
    \Spec_c (T) = \{ z\in \Gr ~:~ |\FF{T} (z)| \ge c |T|\} \,. 
\]
By the definition of the set $\Spec_c (T)$, the H\"older inequality and the Parseval identity, we have 
\[
    \sigma_1 :=
\frac{|B|}{|T|^k N^l} \sum_{z\notin \Spec_c (T)} |\FF{T} (z)|^k
    \sum_{r_1+\dots+r_l=z} |\FF{f} (r_1,\dots,r_l)| |\FF{B}(r_1)| \dots |\FF{B}(r_l)| 
\]
\[ 
    \le 
    \frac{c^k |B|}{N^l} \sum_{r_2,\dots,r_l} 
    |\FF{B}(r_2)| \dots |\FF{B}(r_l)| 
    \sum_{z} |\FF{f} (z-r_2+\dots+r_l,\dots,r_l)| |\FF{B} (z-r_2+\dots+r_l)|
\]
\[
    \le 
    \frac{c^k |B|^{3/2}}{N^{l-1}} \sum_{r_2,\dots,r_l} 
    |\FF{B}(r_2)| \dots |\FF{B}(r_l)| \left( \sum_a |\FF{f}_a (r_2,\dots,r_l)|^2 \right)^{1/2} \,,
\]
    where $f_a(x_2,\dots,x_l) = f(a,x_2,\dots,x_l)$. 
    Using the H\"older inequality and the Parseval formula one more time, we derive 
\[
    \sigma_1 \le \frac{c^k |B|^{(l+2)/2}}{N^{(l-1)/2}}
    \left( \sum_{r_2,\dots,r_l} \sum_a |\FF{f}_a (r_2,\dots,r_l)|^2  \right)^{1/2}
    \le 
    c^k |B|^{(l+2)/2} N^{l/2} \le 2^{-1}\epsilon |B|^{l+1} \,.
\]
    Here we have taken $c = 1/2$ and $k = \lceil 2\L (\epsilon \beta^l) \rceil$, say. 
    For the sum over $z\in \Spec_{c} (T)$ we use the Chang lemma \cite{chang2002polynomial} about the dimension of the spectrum and find a subspace $V$ such that \eqref{f:Sanders_V} takes place and 
\[
    \mathrm{codim} V \ll \log (N/|T|) \ll \epsilon^{-2} l \L^2 (\beta) k^2
    \ll
     \epsilon^{-2} l \L^2 (\beta) \L^2 (\epsilon \beta^l) \,,
\]
see details in \cite{CS}, \cite{sanders2012bogolyubov}, \cite{sanders2013structure} or in \cite[Section 5]{sh_str_survey}. 
This completes the proof. 
$\hfill\Box$
\end{proof}

\section{Some generalizations of Kelley--Meka results}

Using the density increment Kelley--Meka \cite{kelley2023strong} 
(or just repeat the calculations of the previous section, combining forthcoming  Proposition \ref{p:increment_Ekl} in the case $l=2$) 
obtained the following result. 

\begin{theorem}
    Let $\Gr=\F_p^n$, $A\subseteq \Gr$ be a set, $|A|=\d N$, and $\eps>0$ be a parameter.
    Then there is a subspace $V\subseteq \Gr$ and $x\in \Gr$ such that $A\cap (V+x)$ is $\eps$--uniform relatively to $\E^k_2$, $\mu_{V+x} (A) \ge \d$, and 
\begin{equation}\label{f:uniform_KM}
    \codim V \ll \eps^{-14} k^4 \L^3 (\d) \L^{2} (\eps \d) \cdot \L^{4} (\eps) \,.
\end{equation}
\label{t:uniform_KM}
\end{theorem}

\begin{remark}
    Actually, we formulate Theorem \ref{t:uniform_KM} in the form of Bloom--Sisask \cite{bloom2023kelley}. Kelley--Meka \cite{kelley2023strong} obtained this result without $\L^{4} (\eps)$ in \eqref{f:uniform_KM}. 
\end{remark}

We generalize Theorem \ref{t:uniform_KM} for the higher energies $\E^k_l$.

\begin{theorem}
    Let $\Gr=\F_p^n$, $A\subseteq \Gr$ be a set, $|A|=\d N$, and $\eps \in (0,1]$ be a parameter.
    Then there is a subspace $V\subseteq \Gr$ and $x\in \Gr$ such that $A\cap (V+x)$ is $\eps$--uniform relatively to $\E^k_l$, $\mu_{V+x}(A) \ge \d$ and 
\begin{equation}\label{f:uniform_KM_k}
    \codim V \ll 
\eps^{-28l^l} (8l)^{28l^l} k^4 \L^{4l} (\eps) \L^{5l} (\d)
    \,.
\end{equation}
\label{t:uniform_KM_k}
\end{theorem}

Now we are ready to obtain our driving result about the density increment. 
As always we will apply Proposition \ref{p:increment_Ekl} in an iterative way  and we see that estimate \eqref{cond:increment_2_Ekl}  allows us to do it in at most $O(\eps^{-1} \L(\d))$ times. 

\begin{proposition}
    Let $\Gr = \F_p^n$, $A\subseteq \Gr$, $|A|=\d N$, $\eps>0$ be a real number and $k,l \ge 2$ be positive integers, $kl \gg \eps^{-1} \L(\eps)$. 
    Suppose that 
\begin{equation}\label{cond:increment_Ekl}
    \E^k_{l} (A) \ge (1+\eps)^{kl} \d^{kl} N^{k+l} \,,
\end{equation} 
    and that $A$ is $\eps/5$--uniform relatively to $\E^{k_*}_{l-1}$ for an even  $k_* = O(kl \L (\delta))$.
    Then there is a subspace $V \subseteq \Gr$ such that 
\begin{equation}\label{cond:increment_1_Ekl}
    \mathrm{codim} V \ll \eps^{-2} l^3 k^4 \L^2 (\d) \L^2 (\eps \d)
    \,,
\end{equation} 
    and for a certain $x\in \Gr$ one has 
\begin{equation}\label{cond:increment_2_Ekl} 
    |A\cap (V+x)| \ge (1+\eps/8) \d |V| \,.
\end{equation} 
\label{p:increment_Ekl}
\end{proposition}
\begin{proof} 
    Applying Lemma \ref{l:BSzG_k} for the energy $\E^k_{l} (A)$ with the parameters $\epsilon = \eps$, $\eta = \eps/30$, we construct the set 
\[
    S = \{ |x|=l ~:~ \Cf_l (A) (x) \ge (1+\eps/4)^l \d^l N \}
\]
and such that for a certain set $B \subseteq \Gr$,  $|B| > 2^{-1/(l-1)} (1+\eps)^k \d^k N := \beta N$ the following holds 
\begin{equation}\label{tmp:01.03_1_Ekl} 
    N^{-1} \sum_{|x|=l} S (x) \Cf_l (B) (x) \ge (1-2\eta) |B|^l \,.
\end{equation} 
    We have $kl \ge 4\eps^{-1} \L (\eta)$ and thus Lemma \ref{l:BSzG_k} can be applied indeed. 
    Using formulae \eqref{f:E^k_l_symmetries} and making the required change of the variables (in \eqref{def:Cf} we put, consequently, $z\to z-x_1$) one can see that \eqref{tmp:01.03_1_Ekl} is equivalent to 
\begin{equation}\label{tmp:01.03_2_Ekl} 
    \sum_{|x|=l-1} \bar{S} (x) (B^{l-1} \circ \mD_{l-1} (B)) (x)  \ge (1-2\eta) |B|^l \,,
\end{equation} 
    where $\bar{S} \subseteq \Gr^{l-1}$ is a certain set which is constructed via the set $S$, see formulae \eqref{f:E^k_l_symmetries}, \eqref{f:E^k_l:C,C'}. 
    Now we apply Lemma \ref{l:Sanders_V} with $f=\bar{S}$, $B=B$, $\epsilon   = \eta$, and $l=l-1$.
    By this result and inequality \eqref{tmp:01.03_2_Ekl} we find a subspace $V\subseteq \Gr$ such that 
    the co--dimension of $V$ is controlled 
    by estimate 
    \eqref{f:Sanders_V_codim} and 
\begin{equation}\label{tmp:01.03_3_Ekl} 
    \sum_{|x|=l-1} \bar{S} (x) (B^{l-1} \circ \mD_{l-1} (B * \mu_V)) (x)  \ge (1-3\eta) |B|^l \,.
\end{equation} 
    By the definition of the set $S$ (and hence $\bar{S}$), we have $\Cf_l (A) (x) \ge (1+\eps/4)^l \d^l N$
    for any $x \in \bar{S}$ and hence inequality \eqref{tmp:01.03_3_Ekl}  give us 
\[
    (1-3\eta) (1+\eps/4)^l \d^l |B|^l N \le \sum_{|x|=l-1} 
    (A^{l-1} \circ \mD_{l-1} (A)) (x)
    (B^{l-1} \circ \mD_{l-1} (B * \mu_V)) (x)
\]
\[
    =
    \sum_{\a} (A*B*\mu_V) (\a)  
    \sum_{|x|=l-1} A^{l-1} (x) B^{l-1} (x + \mD_{l-1} (\a)) 
    = \sum_{\a} (A*B*\mu_V) (\a) (A\circ B)^{l-1} (\a) 
\]
\begin{equation}\label{tmp:01.03_4_Ekl} 
\le 
    |B| \| A * \mu_V \|_\infty  \sum_{\a}  (A\circ B)^{l-1} (\a) 
\,.
\end{equation} 
    Now by our assumption $A$ is $\eps/5$--uniform relatively to $\E^{k_*}_{l-1}$ and a certain  even $k_* = O(kl \L (\delta))$.
    Using Corollary \ref{c:A_circ_B}, we derive
\[
    \| A * \mu_V \|_\infty \ge \d (1-3\eta) (1+\eps/4) \ge \d (1+\eps/8) \,.
\]
    Finally, thanks to \eqref{f:Sanders_V_codim}, we get
\begin{equation}\label{tmp:codim}
    \codim V \ll \eps^{-2} l \L^2 (\beta) \L^2 (\eps \beta^l) 
    \ll 
     \eps^{-2} l^3 k^4 \L^2 (\d) \L^2 (\eps \d) \,.
\end{equation} 
This completes the proof. 
$\hfill\Box$
\end{proof}

\bp 

Now we  
can prove our new Theorem \ref{t:uniform_KM_k}.

As always 
the proof follows the density increment scheme and our aim is to construct a shift of a subspace $V_\eps (l,k)$ where the set $A$ is $\eps$--uniform relatively to $\E_l^k$. 
Also, to obtain Theorem \ref{t:uniform_KM_k} we use induction on parameter $l\ge 2$ and the first step of the induction for $l=2$ and an arbitrary $k$ one can use either  
Kelley--Meka  Theorem \ref{t:uniform_KM}
or the arguments of 
Section \ref{sec:E^k_l} (one can check that we do not need any uniformity conditions in this case),  
combining with  Proposition \ref{p:increment_Ekl}. 
Now let $l\ge 3$ and suppose that the set $A$ is not $\eps$--uniform relatively to $\E^k_{l}$ for a certain $k$ because otherwise there is nothing to prove. 
Of course (see, e.g., inequality \eqref{f:Holder_Ekl}), one can take $k$ to be a sufficiently large number and we choose $k$ such that $k \gg l\eps^{-l} \L (\eps)$.  
Put $\eps_l=\frac{\eps^l}{8l}$.
We can freely assume that our set $A$ is $\eps_l/5$--uniform relatively to $\E^{k_*}_{l-1}$ with 
$k_* = k_* (l) =O(k l^2 \eps^{-l} \L(\eps) \L(\d))$ in a shift of a subspace $V_{\eps_l/5} (l-1,k_*(l))$ 
thanks to Theorem \ref{t:uniform_KM} in the case $l=3$ or by the induction assumption for larger $l$. 
Now we apply Lemma \ref{l:e_to_d} and find $k_1 = O(kl \eps^{-l} \L(\eps))$ such that  
\[
    \E_l^{k_1} (A) \ge \left(1+ \eps_l \right)^{lk_1} \d^{lk_1} N^{l+k_1} \,.
\]
After that we use Proposition \ref{p:increment_Ekl} with 
$\eps = \eps_l$.
One can check that  $kl \gg \eps_l^{-1} \L(\eps_l)$ and that $A$ is sufficiently uniform set (to apply our proposition) relatively to $\E^{k_*}_{l-1}$ thanks to our choice of $k_*$.  
Estimate \eqref{cond:increment_2_Ekl}  implies that 
the procedure must stop after $O(\eps^{-1}_l \L(\d))$ number of steps and thus the final co--dimension is  
\[
   \codim V_\eps (l,k) \ll 
   \eps^{-1}_l \L(\d) \codim V_{\eps_l/5} (l-1,k_* (l)) \ll 
   l \eps^{-l} \L(\d) \codim V_{\eps_l/5} (l-1, k l^2 \eps^{-l} \L(\eps) \L(\d)) \,.
\]   
    Put $L=l!$. 
    Solving the functional inequality above and using \eqref{f:uniform_KM}, we get 
\begin{equation}\label{tmp:09.03_1}
    \codim V_\eps (l,k) \ll L \eps^{-l(l-1)/2} \L^{l-2} (\d) \codim V_{(\eps/8l)^{2L}} (2, k(l!)^2  \eps^{-l(l-1)/2} (\L(\eps) \L(\d))^{l-2})
\end{equation} 
\[
    \ll \eps^{-28l^l} (8l)^{28l^l} k^4 \L^{4l} (\eps) \L^{5l} (\d)
\]
as required. 
Actually, one can see that the number of steps of our algorithm is at most $O(\eps^{-1}_l \L(\d))$ (due to every time we increase the density $\d$ to $\d (1+\eps_l/8)$ and hence we do not need the first multiple in \eqref{tmp:09.03_1}). Nevertheless, it gives us a bound of the same sort.   
This concludes the proof of the theorem. 
$\hfill\Box$

\bp

Theorem \ref{t:uniform_KM_k} allows us to find a shift of subspace where our set $A$ is uniform relatively to $\E^k_l$. 
A modification of this argument allows to prove more
and this is important for applications. 
Actually, for any relatively dense $A\subseteq \F_p^n$, $|A| = \d N$ there is a partition of $\F_p^n$ onto shifts of subspaces $V_j$ such that 
\begin{equation}\label{f:A_L2}
    A = \left(\bigsqcup_j A\cap (V_j+ x_j) \right) \bigsqcup \Omega 
\end{equation} 
such that $|\Omega|= \omega N$, each set $(A-x_j)$ is $\eps$--uniform in $V_j$ relatively to $\E^k_l$ and $\codim V_j = O_{\eps,\omega, \d} (1)$. 
Such results are not new see, e.g.,  \cite{GT_primes} and \cite{Green_models}, \cite{sh_Sz_LMS},  \cite{sh_Sz_survey}.
The idea is to replace $L_\infty$--increment to $L_q$--increment with a controllable parameter $q$.  
In this section we obtain the driving proposition, which we will use in the future paper. 

\begin{proposition}
    Let $\Gr = \F_p^n$, $A\subseteq \Gr$, $|A|=\d N$, $\eps \in (0,1/2)$ be a real number and $k,l \ge 2$ be positive integers, and $k \gg l\eps^{-l} \L(\eps)$.
    Suppose that 
\begin{equation}\label{cond:L2_ increment_Ekl}
    \E^k_{l} (f_A) = \eps^{kl} \d^{kl} N^{k+l} \,,
\end{equation} 
    and that $A$ is $2^{-11l} l^{-l} \eps^{l^2}$--uniform with respect to $\E^{k_*}_{l-1}$ for  an even $k_* = O(kl^2 \eps^{-l} \L(\eps) \L (\delta))$.
    Then there is a subspace $V \subseteq \Gr$ such that 
\begin{equation}\label{cond:L2_increment_1_Ekl}
    \mathrm{codim} V \ll \eps^{-2l} l^5 k^4 \L^2 (\d) \L^2 (\eps \d)
    \,,
\end{equation} 
\begin{equation}\label{cond:L2_increment_2_Ekl+} 
    N^{-1} \sum_x (f_A * \mu_V)^l (x) \ge  
    2 \left(\frac{\eps \d}{2^7} \right)^l
    \,.
\end{equation} 
    In particular, 
\begin{equation}\label{cond:L2_increment_new} 
    N^{-1} \sum_x (A * \mu_V)^l (x) \ge  
    \d^l \left(1+\frac{\eps^l}{2^{7l}} \right)
    \,.
\end{equation}
    Further suppose that $A$ is $2^{-11} \eps^l/l^2$--uniform with respect to $\E^{k_*}_{l-1}$ for an even  $k_* = O(kl^2 \eps^{-l} \L(\eps) \L (\delta))$. 
    Then for a certain even $q$, $q \ll l^3 \L(\eps \delta)$ and $l>2$ the following holds 
\begin{equation}\label{cond:L2_increment_2_Ekl} 
    N^{-1} \sum_x (f_A * \mu_V)^q (x) \ge  
    (4\d)^{q} \,.
\end{equation} 
    In particular,
\begin{equation}\label{cond:L2_increment_2_Ekl_A} 
    N^{-1} \sum_x (A * \mu_V)^q (x) \ge  (2\d)^{q} \,.
\end{equation} 
\label{p:L2_increment_Ekl}
\end{proposition}
\begin{proof} 
    Using Lemma \ref{l:e_to_d} we find
    $k_1 \ll kl \eps^{-l} \L(\eps)$ such that 
\[
    \E^{k_1}_{l} (A) \ge 
    \left(1+\frac{\eps \eps^{l-1}_*}{8l} \right)^{k_1 l} \d^{k_1 l} N^{k_1+l}
    =
    (1+\eps_l)^{k_1 l} \d^{k_1 l} N^{k_1+l} \,.
\]
    Let $\zeta = 2^{-8l} \eps^l_l$, $\zeta_1 = 2^{-8} \eps_l/l$ and  for any positive integer $n$ we put 
    $S (n) = N^{-1} \sum_\a (f_A * \mu_V)^n (\a)$. 
We use the same argument as in Proposition \ref{p:increment_Ekl} (the parameter $\eta$ equals  $\eps_l/30$)
before inequality \eqref{tmp:01.03_4_Ekl}.
One has 
\begin{equation}\label{tmp:03.09_2}
    (1-3\eta) (1+\eps_l/4)^l \d^l |B|^l N = \d |B| \sum_\a (A\circ B)^{l-1} (\a) + \sum_\a (f_A * \mu_V) (\a) ((A\circ B)^{l-1} \circ B) (\a) = \sigma_1 +\sigma_2 \,.
\end{equation} 
    Here we have the required upper bound for the dimension of $V$ similar to \eqref{tmp:codim} 
\begin{equation*}
    \codim V \ll \eps^{-2}_l l \L^2 (\beta) \L^2 (\eps_l \beta^l) 
    \ll 
     \eps^{-2l} l^5 k^4 \L^2 (\d) \L^2 (\eps \d) \,.
\end{equation*}
    We need to consider the cases $l=2$ and $l>2$ separately. 
    In the case $l=2$ we follow the method from \cite{kelley2023strong}.
    Clearly, if $l=2$, then $\sigma_1=\d |B|^2 |A|$ and hence
\begin{equation}\label{tmp:11.03_1}
    \frac{7\eps_2}{18} \d^2 |B|^2 N \le \sum_\a ((f_A \circ  f_A) * \mu_V) (\a) (B \circ B) (\a) \,.
\end{equation} 
    Using  the Fourier transform twice and the fact that $\FF{\mu}_V (r) = V^{\perp} (r)$, we get
\[
     \frac{7\eps_2}{18} \d^2 |B|^2 N 
     \le 
     N^{-1} \sum_{r} \FF{\mu}_V (r) |\FF{f}_A (r)|^2 
     |\FF{B} (r)|^2
     \le 
     N^{-1} |B|^2 \sum_{r} \FF{\mu}_V (r) |\FF{f}_A (r)|^2 =
     |B|^2 \sum_{\a} (f_A\circ f_A) (\a) \mu_V (\a) \,.
\]
    It remains to notice that $\mu_V \circ \mu_V = \mu_V$. Let us give another proof which does not use the Fourier approach. Using the identity $\mu_V \circ \mu_V = \mu_V$ we estimate the right--hand side of 
    \eqref{tmp:11.03_1} as 
\[
    \| ((f_A \circ  f_A) * (\mu_V \circ \mu_V)) (\a) \|_\infty \cdot 
    \sum_\a  (B \circ B) (\a) 
    = |B|^2  \sum_x (f_A * \mu_V)^2(x)  
\]
as required.

    Now consider the general case. 
    To calculate the sum $\sigma_1$ we apply Corollary \ref{c:A_circ_B} and our assumption on $l$ to derive 
\begin{equation}\label{tmp:03.09_3}
    (\d |B|)^{-1} \sigma_1 \le \d^{l-1} |B|^{l-1} N (1+1.25 \zeta)^{l-1} 
    \le \d^{l-1} |B|^{l-1} N ( 1 + 2\zeta l) 
\,.
\end{equation} 
    In view of our choice of $\eta$ and $\zeta \le \zeta_1 = 2^{-8} \eps_l/l$ it gives us
\begin{equation}\label{tmp:03.09_3'}
    2^{-3} \eps_l l \d^l |B|^l N
    \le 
    ((1-3\eta)(1+\eps_l l/4)- 1-2\zeta l)
    \d^l |B|^l N
    \le 
    \sigma_2
   \,.
\end{equation}  
    Applying the H\"older inequality several times for even $l$, we derive
\[
    \sigma_2 = \sum_\a (f_A * \mu_V) (\a) ((A\circ B)^{l-1} \circ B) (\a)
\]
\begin{equation}\label{tmp:S(l)_calculation-}
    \le 
     \left(\sum_\a (f_A * \mu_V)^l (\a) \right)^{1/l}
     \left( \sum_\a ((A\circ B)^{l-1} \circ B)^{l/(l-1)} (\a) \right)^{1-1/l}
\end{equation} 
\begin{equation}\label{tmp:S(l)_calculation}
     \le
     |B| S^{1/l} (l) N^{1/l} 
     \left( \sum_\a (A\circ B)^{l} (\a) \right)^{1-1/l}  N^{1-1/l} \,.
\end{equation} 
    Now using Corollary \ref{c:A_circ_B} and recalling estimate \eqref{tmp:03.09_3'}, we see that 
\begin{equation}\label{tmp:19.03_even}
    2^{-4l+2} l^l \eps^l_l \d^l \le S(l) = N^{-1} \sum_\a (f_A * \mu_V)^l (\a)
\end{equation} 
    as required. 
    For odd $l$ the argument is the similar, one has 
\[
    \sum_\a |(f_A * \mu_V)^{l} (\a)| 
    \le 
    \sum_\a (f_A * \mu_V)^{l-1} (\a) ( (f_A * \mu_V) (\a) + 2\d)
    =
    N S(l) + 2\d \sum_\a (f_A * \mu_V)^{l-1} (\a)
\]
\[
    \le 
    N S(l) + 2 \zeta^{l-1} \d^{l} N
    \le  N S(l) + 2^{-4l} l^l \eps^l_l \d^l N
\]
    and hence we have obtain a lower bound for $S(l)$ similar to \eqref{tmp:19.03_even}. 
    Thus we have 
    proved 
    \eqref{cond:L2_increment_2_Ekl+}. 
    To derive  \eqref{cond:L2_increment_new}, we write 
\[
    N^{-1} \sum_{\a} (A* \mu_V)^l (\a) = S(l) + \d^l +
    \sum_{0<j<l} \binom{l}{j} \d^{l-j} \sum_{\a} (f_A* \mu_V)^j (\a) 
\]
    and apply Lemma \ref{l:e_to_d}
\[
    N^{-1} \sum_{\a} (A* \mu_V)^l (\a) \ge S(l) + \d^l - 2l \d^l \zeta \ge \d^l (1+2^{-4l} l^l \eps^l_l) 
\]
    as required. 

\bigskip 

Now let us obtain \eqref{cond:L2_increment_2_Ekl} and  \eqref{cond:L2_increment_2_Ekl_A}.  
    Returning to \eqref{tmp:03.09_2} and \eqref{tmp:03.09_3'}, we have 
\[
    \sigma_2 =\sum_{j=1}^{l-1} \binom{l-1}{j} (\d |B|)^{l-1-j} \sum_\a ((f_A \circ B)^j \circ B) (\a) (f_A * \mu_V)  (\a)
\]
\begin{equation}\label{tmp:11.03_2}
    = \sum_{j=1}^{l-1} \binom{l-1}{j} (\d |B|)^{l-1-j} \sigma(j) \,.
\end{equation}  
    For even $j$ we can apply formula \eqref{tmp:S(l)_calculation-} with $l=q$  
    again, then the estimate $\| ((f_A \circ B)^j \circ B) (\a) \|_\infty \le |B|^{j+1}$, 
    as well as 
    Lemma \ref{l:uniformity} 
    with the  parameter $j<l$ and derive 
\[
    \sigma^q (j) \le |B|^{q+j} \left( \sum_\a (f_A \circ B)^{j} (\a) \right)^{q-1}
    \cdot 
    \sum_\a (f_A * \mu_V)^q  (\a) 
\]
\begin{equation}\label{tmp:12.03_1}
    \le \zeta^{j(q-1)}_1 \d^{j(q-1)} |B|^{(j+1)q} N^{q}  \cdot S(q)
\end{equation}  
For odd $j$, we write 
\[
    \sigma (j) = \sum_\a ((f_A \circ B)^{j-1} \cdot (A\circ B) \circ B) (\a) (f_A * \mu_V)  (\a) 
    - \d |B| \sigma(j-1) 
    = \sigma_* (j) - \d |B| \sigma(j-1) 
\]
    and thus we need to estimate $\sigma_* (j)$   because in  $\sigma(j-1)$  the argument $j-1$ is even.
    To bound the sum 
    $\sigma_* (j)$  
    we use the same argument as in \eqref{tmp:12.03_1} (and the identity  $(A\circ B) (\a) = (f_A\circ B) (\a) + \d |B|$, of course) to show that 
\[
    \sigma^q_* (j) \le |B|^{q+j} \left( \sum_\a (f_A \circ B)^{j-1} (\a) ( A\circ B) (\a) \right)^{q-1}
    \cdot 
    \sum_\a (f_A * \mu_V)^q  (\a) 
\]
\begin{equation}\label{tmp:12.03_1'}
    \le 2^{q-1} \zeta^{j(q-1)}_1 \d^{j(q-1)} |B|^{(j+1)q} N^{q}  \cdot S(q) \,.
\end{equation}  
    Thus in view of \eqref{tmp:03.09_3'}, \eqref{tmp:12.03_1}, \eqref{tmp:12.03_1'} as well as our choice of $q$, we get 
\[
    2^{-5} \eps_l l \d \le l((1+\zeta_1)^{l-1} - 1) S(q)^{1/q} 
    \le 2\zeta_1 l^2 S(q)^{1/q} 
\]
    and hence 
\begin{equation}\label{tmp:4delta} 
    (4\d)^{q} \le  N^{-1} \sum_\a (f_A * \mu_V)^q (\a)
\end{equation} 
as required. 
To get \eqref{cond:L2_increment_2_Ekl_A} just use the identity  $(f_A * \mu_V) (\a) = (A * \mu_V) (\a) - \d$ and apply the binomial formula to \eqref{tmp:4delta}.
%
%
This completes the proof. 
$\hfill\Box$
\end{proof}

\begin{remark}
    At first sight 
    it looks strange that estimates \eqref{cond:L2_increment_2_Ekl}, \eqref{cond:L2_increment_2_Ekl_A} do not depend on $\eps$ but we just pay for this taking larger codimension   in \eqref{cond:L2_increment_1_Ekl}. 
\label{r:codim_absorbe}
\end{remark}

\section{Appendix}
\label{sec:appendix}

\subsection{On longer progressions}
\label{subs:AP4}

We need a generalization of norms \eqref{def:E_kl}. Having vectors $x=(x_1,\dots,x_s) \in \Gr^s$ and $y=(y_1,\dots,y_t) \in \Gr^t$ (we write that $|x|=s$ and $|y|=t$) define its ``Minkowski'' sum as $x\oplus y \in \Gr^{st}$, where the components of $x\oplus y$ are $x_i+y_j$, $i\in [s]$, $j\in [t]$ (and similarly for higher sums). 
Put
\begin{equation}\label{def:mE_kl}
    \mathcal{E}^k_{s,t} (f) = \sum_{|x|=s}\, \sum_{|y|=t} \Cf^k_{st} (f) (x\oplus y) =
    \sum_{|x|=s}\, \sum_{|y|=t}\, \sum_{|z|=k} \mP_{stk} (f) (x\oplus y \oplus z) 
    \,. 
\end{equation} 
In these terms 
\begin{equation}\label{def:mE_kl_pre}
    \E^k_l (f) =  \sum_{|x|=l}\, \sum_{|y|=k}\, \mP_{kl } (f) (x\oplus y) \,. 
\end{equation}
For even $k,s,t$ and a real function  $f$ one has $\mathcal{E}^k_{s,t} (f) \ge \E^k_{s} (f), \E^k_{t} (f) \ge 0$ and the triangle inequality for $\mathcal{E}^k_{s,t}$ can be obtained exactly as in \cite[Appendix]{sh_Ek} or just see Subsection \ref{subs:gen_norms} (but of course one needs an additional application of the H\"older inequality due to we have the longer sum in \eqref{def:mE_kl} than in \eqref{def:mE_kl_pre}). 
Thus 
$(\mathcal{E}^k_{s,t} (f))^{1/kst}$ 
defines a norm of $f:\Gr \to \R$ in the case of even $k,s,t$. 
Notice that similar to $\E^k_l (f)$ the quantity $\mathcal{E}^k_{s,t} (f) \ge 0$, 
provided at least one of $k,s,t$ is even but, nevertheless,  it is not  always a norm in this case, see \cite[Section 4]{sh_Ek}.  
By some symmetricity  reasons (see, e.g., formulae \eqref{f:E^k_st_symmetries} below) we make a normalization and put 
$$
\| f\|_{\mathcal{E}^k_{s,t}} := (|\Gr|^{-2} \mathcal{E}^k_{s,t} (f))^{1/kst} := (\bar{\mathcal{E}}^k_{s,t} (f))^{1/kst} 
$$
for $f:\Gr \to \R$. 
Clearly, one has 
\begin{equation}\label{f:mE_kl_1}
    \mathcal{E}^k_{s,t} (f)
    = 
    \sum_{|x|=s}\, \sum_{|z|=k}  \Cf^{t}_s (f_z) (x) 
    =
    \sum_{|x|=s}\, \sum_{|z|=k}  \Cf^{t}_{sk} (f) (x\oplus z) \,,
\end{equation}
    and
\begin{equation}\label{f:mE_kl_2}    
    \mathcal{E}^k_{s,t} (f)
    =
    \sum_{|y|=t}\, \sum_{|z|=k}  \Cf^{s}_t (f_z) (y) 
    =
    \sum_{|y|=t}\, \sum_{|z|=k}  \Cf^{s}_{tk} (f) (y \oplus z) 
    \,.
\end{equation} 
Thus we have the duality relation similar to \eqref{def:E_kl}
\begin{equation}\label{f:E^k_st_duality}
    \mathcal{E}^k_{t,s} (f) = \mathcal{E}^k_{s,t} (f) = \mathcal{E}^t_{s,k} (f) = \mathcal{E}^s_{t,k} (f)\,. 
\end{equation} 
Also, let us remark that the expectations over $x\oplus y$ of the generalized convolution of any real function $f : \Gr \to \mathbb{R}$ is connected with the higher energies  
\begin{equation}\label{f:E^k_st_expectation}
    \sum_{|x|=s}\, \sum_{|y|=t} \Cf_{st} (f) (x\oplus y) = N \E^t_{s} (f) \,.
\end{equation}
In particular, the expectation above is always non--negative if $s$ or $t$ is an even number and we see immediately that the duality \eqref{def:E_kl} takes place. Formula \eqref{f:E^k_st_expectation} can be proved directly or it follows from  \eqref{f:mE_kl_1}, \eqref{f:mE_kl_2} and the fact that $\mathcal{E}^k_{t,1} (f) = N \E^k_{t} (f)$.
Finally, notice that in contrast to $\Cf_l (x)$ the function $\Cf_{st}(x\oplus y)$ enjoys even two symmetries, namely,
\begin{equation}\label{f:E^k_st_symmetries}
    \Cf_{st} (f) (x\oplus y) = \Cf_{st} (f) ((x+\mD_s (w_1))\oplus y) = 
    \Cf_{st} (f) (x\oplus (y + \mD_t(w_2))) 
\end{equation} 
for any $w_1,w_2\in \Gr$. 
It gives, in particular, 
\begin{equation}\label{f:Cf_to_Cf'+}
    \Cf_{st} (f) (x\oplus y) = \Cf_{st} (f) ((x - \mD_s (x_1))\oplus (y-\mD_t (y_1))) =  N^2 \Cf'_{st} (f) (w) \,,
\end{equation} 
where $|w|=st-1$ and, 
more concretely, 
$w_{ij} =(x_i-x_1)+(y_j-y_1)$, $i\in [s]$, $j\in [t]$ and $(i,j) \neq (1,1)$.

\bp 

Now we are ready to obtain our counting lemma. 
Let us 
write $L(x,y) = \a x+\beta y + \gamma$ for a non--trivial linear form.
We say that two forms are non--proportional if their coefficients are not proportional. 
Given  
a real number $q>1$ put $q^* = \frac{q}{q-1}$. 

\begin{theorem}
    Let $N$ be a prime number and 
$k=4$, $l_1,l_2 \ge 2$ be positive integers.
    Also, let $f_1,\dots, f_k : \Z/N\Z \to \R$ be functions  and $L_1,\dots,L_k$ be non--proportional linear forms such that $L_2,\dots,L_k$ depend on both variables. Then  
\begin{equation}\label{f:counting}
    \left|\sum_{x,y} f_1( L_1(x,y)) \dots f_k( L_k(x,y)) \right| 
    \le 
    \| f_1\|_{l^*_1} \| f_2\|_{l^*_2}
    \| f_3 \|_{\mathcal{E}^2_{l_1,l_2}} 
    \| f_{4} \|_{\mathcal{E}^2_{l_1,l_2}} \,.
\end{equation}
\label{t:counting}
\end{theorem}
\begin{proof} 
    Let $\sigma$ be the left--hand side of \eqref{f:counting}. 
    Without loss of  generality one can assume that $L_j(x,y)=\a_j x + \beta_j y$, $j\in [k]$. 
    Consider the nonzero form $L_1$ and suppose for concreteness that  $\a_1\neq 0$. 
    Changing the variables $\a_1 x + \beta_1 y \to x$, we obtain
\begin{equation}\label{tmp:sigma1}
    \sigma = \sum_{x,y} f_1(x) f_2( \tilde{L}_2(x,y)) \dots f_k( \tilde{L}_k(x,y)) \,,
\end{equation} 
    where here and below we write $\tilde{L}_j(x,y) =  L_{j} (x,y) = \a_j x + \beta_j y$ and the coefficients  $\a_j$, $\beta_j$  may change from line to line.
    Anyway one can check that all new forms $\tilde{L}_2,\dots, \tilde{L}_k$ in \eqref{tmp:sigma1} are nonzero and non--proportional.
    Moreover, by assumption the initial forms $L_2,\dots,L_k$ depend on both variables and we see that the new forms in \eqref{tmp:sigma1} depend on both variables as well. 
    Now we use the H\"older inequality and get
\[
    (\sigma/\|f_1\|_{l^*_1})^{l_1}
    \le
    \sum_{x} \left( \sum_y f_2( L_2(x,y)) \dots f_k( L_k(x,y)) \right)^{l_1}
\]
\[
    =
    \sum_{x,y} \mP_{l_1} (f_2) ( L_2(\mD_{l_1} (x),\mP_{l_1} (y))) \dots \mP_{l_1} (f_k) ( L_k(\mD_{l_1} (x),\mP_{l_1} (y))) \,.
\]
    Notice that we have decreased the number of our linear forms (but increased the number of variables). 
    Now let us make the  changing of the variables  similar to above, namely, $\a_1 \mD_{l_1} (x) + \beta_1 \mP_{l_1} (y) \to \mP_{l_1} (y)$ and again one can easily check that we preserve all conditions on our linear forms $L_3, \dots, L_k$. 
    Thus one has 
\[
    (\sigma/\|f_1\|_{l^*_1})^{l_1} \le 
    \sum_y \mP_{l_1} (f_2) (\mP_{l_1} (y)) \sum_x \mP_{l_1} (f_3) ( L_3(\mD_{l_1} (x),\mP_{l_1} (y))) \dots \mP_{l_1} (f_k) ( L_k(\mD_{l_1} (x),\mP_{l_1} (y)))
\]
    and using the H\"older inequality 
    one more time, 
    as well as the obvious identity 
\begin{equation}\label{tmp:26.02_-1}
    \left(\sum_y \mP^{l^*_2}_{l_1} (f_2) (\mP_{l_1} (y))  \right)^{l_2-1}  = \| f_2 \|^{l_1 l_2}_{l^*_2} \,, 
\end{equation} 
    we derive
\[
    (\sigma/\|f_1\|_{l^*_1} \|f_2\|_{l^*_2})^{l_1 l_2 } 
\]
\begin{equation}\label{tmp:sigma_k_2}
    \le 
    \sum_{x,y} \mP_{l_1 l_2} (f_3) ( L_3(\mP_{l_2} \mD_{l_1} (x),\mD_{l_2} \mP_{l_1} (y))) \mP_{l_1 l_2} (f_4) ( L_4(\mP_{l_2} \mD_{l_1} (x),\mD_{l_2} \mP_{l_1} (y))) \,.
\end{equation} 
    Now 
    let us analyse the right--hand side of formula \eqref{tmp:sigma_k_2}. 
    First of all, it is easy to see that there are $l_2$ different variables $x_i$ and  $l_1$ different variables $y_j$ in 
    \eqref{tmp:sigma_k_2}. 
    Secondly, take the form $L_{k-1}$ (for $L_k$ the argument is the same) and notice that it depends on $\a_{k-1} x_i + \beta_{k-1} y_j$, $i\in [l_2]$, $j\in [l_1]$ and that every such expression appears exactly once. 
    Now introducing two more variables $z,w$ such that $x_i\to x_i+z$, $y_j \to y_j + w$ and then replacing $z,w$ to other variables $Z,W$, where $Z=\a_{k-1} z + \beta_{k-1} w$, $W = \a_{k} z + \beta_{k} w$ (this change of the variables is allowable because the forms $L_{k-1}$, $L_k$ are not proportional), we arrive to the quantities  $\Cf_{l_1l_2} (f_{k-1})$, 
    $\Cf_{l_1l_2} (f_{k})$ in 0
    \eqref{tmp:sigma_k_2}.
    Writing  $x= (x_1,\dots, x_{l_2})$, $y= (y_1,\dots, y_{l_1})$, we have finally 
\[
    (\sigma/\|f_1\|_{l^*_1} \|f_2\|_{l^*_2})^{l_1 l_2 } 
    \le 
    N^{-2}
    \sum_{\vec{x},\vec{y}} 
    \Cf_{l_1 l_2} (f_{k-1}) (\a_{k-1} \cdot x \oplus   \beta_{k-1} \cdot y)
    \Cf_{l_1 l_2} (f_{k}) (\a_{k} \cdot x \oplus  \beta_{k} \cdot y) \,.
\]
    Using the H\"older inequality the last time, as well as the fact that $\a_{k-1},\a_k,\beta_{k-1}, \beta_k \neq 0$, we obtain 
\[
 (\sigma/\|f_1\|_{l^*_1} \|f_2\|_{l^*_2})^{l_1 l_2 } 
\]
\[
    \le 
    \left( N^{-2}\sum_{|x| =l_2,\, |y|=l_1}  
    \Cf^2_{l_1 l_2} (f_{k-1}) (x \oplus y)
    \right)^{1/2} 
    \left( N^{-2} \sum_{|x| =l_2,\, |y|=l_1}  \Cf^2_{l_1 l_2} (f_{k}) (x \oplus y)
    \right)^{1/2} 
\]
\[    
    = 
    \| f_{k-1} \|^{l_1 l_2}_{\mathcal{E}^2_{l_1,l_2}} \cdot 
    \| f_{k} \|^{l_1 l_2}_{\mathcal{E}^2_{l_1,l_2}}
\]
as required. 
$\hfill\Box$
\end{proof}

\begin{remark}
    One can check that for any $l_1,\dots, l_{k-2}$ one has 
\[
    \frac{1}{l^*_1} + \frac{1}{l^*_2} +\frac{1}{l_1} + \frac{1}{l_2}= 2
\]
    and hence the right--hand side of bound \eqref{f:counting} has the correct order in $N$. 
    Similarly, taking $f_j (x) = f_A (x)$, $j\in [4]$ be the balanced function of a set $A$ and $l_1 \sim l_2 \sim \L (\d)$ we see that the dependence on $\d$ is also correct. 
\end{remark}

\begin{remark}
    As we have said in the previous remark the optimal dependence on the parameters $l_1,l_2$ in Theorem \ref{t:counting} is $l_1 \sim l_2 \sim \L (\d)$. 
    Suppose that the dependence on $\eps$ in Lemma \ref{l:e_to_d} and in all statements below is almost optimal, say, $c\eps$ for a constant $c\in (0,1)$. 
    Thanks to the induction scheme of the proof, it gives us the multiple $c^{\L(\d)} = \d^{-C}$ for a certain $C>0$ in codimention of the subspace $V$, where our set $A$ is uniform.
    But $\d^{-C}$ is more or less that usual Gowers' method gives to us and hence we have no special gain. 
    Thus there are considerable difficulties to extend the Kelley--Meka method for more complicated objects than arithmetic progressions of length three even on  the technical level. 
\label{r:ET_fail}
\end{remark}

We conclude the appendix showing that the convolutions $\Cf_{st} (f) (x\oplus y)$, $|x|=s$, $|y|=t$ 
enjoy the almost periodicity properties similar to the ordinary convolutions $\Cf_s (f) (x)$. 
For $x=(x_1,\dots,x_r)$ let us write  for convenience $\bar{x} = (0,x_1,\dots,x_r)$.

\begin{lemma}
    Let $\eps \in (0,1]$ be a real number, $s,t,q\ge 2$ be positive integers, $l:=st-1$, $B\subseteq \Gr$, $|B| = \beta N$ and $F: \Gr^l \to \R$.
    Then there is a set $T\subseteq \Gr$, $|T|\ge |B| \exp(-O(\eps^{-2} q \log (1/\beta)))$    and such that for any $t\in T$ one has 
\[
    \sum_{|x|=s-1}\, \sum_{|y|=t-1} \left| (F\circ \mD_{l} (B+t) )(\bar{x}\oplus \bar{y}) - (F\circ \mD_{l} (B)) (\bar{x}\oplus \bar{y}) \right|^q 
\]
\begin{equation}\label{f:CS_new}
    \le 
    \eps^q |B|^{q-1} \sum_{|y|=t-1} \Cf'_t (|F|^q)^{s-1} (y) \cdot \Cf'_t (B,|F|^q, \dots, |F|^q) (y) \,.
\end{equation} 
\label{l:CS_new}
\end{lemma} 
\begin{proof} 
    We choose $k = O(\eps^{-2}q)$  random points $b_1,\dots, b_k \in B$ uniformly and independently and let $Z_j ((\bar{x}\oplus \bar{y}) = F((\bar{x}\oplus \bar{y}) + \mD_l (b_j)) - (F\circ \mD_{l} (\mu_B)) (\bar{x}\oplus \bar{y})$. 
    Clearly, the random variables $Z_j$ are independent, have zero expectation and their variances do not exceed $(|F|^2\circ \mD_{l} (\mu_B)) (\bar{x}\oplus \bar{y})$.
    By the Khintchine inequality for sums of independent random variables,
$$
    \| \sum_{j=1}^k Z_j (\bar{x}\oplus \bar{y}) \|_{L_p (\mu^k_B)} \ll (|F|^2\circ \mD_{l} (\mu_B)) (\bar{x}\oplus \bar{y})^{1/2} \,.
$$
Raising the last inequality to the power $q$, dividing by $k^q$, summing over  $\bar{x}\oplus \bar{y}$, and using the H\"older inequality, which gives $(|F|^2\circ \mD_{l} (\mu_B)) (\bar{x}\oplus \bar{y})^{q/2} \le (|F|^q\circ \mD_{l} (\mu_B)) (\bar{x}\oplus \bar{y})$, we get that
$$
    \sum_{|x|=s-1}\, \sum_{|y|=t-1} \int  \left| 
    \frac{1}{k} \sum_{j=1}^k F((\bar{x}\oplus \bar{y}) + \mD_l (b_j)) - (F\circ \mD_{l} (\mu_B)) (\bar{x}\oplus \bar{y})
     \right|^q  d\mu_B^k (x_1,\dots,x_k)
$$
$$
        \ll
        (qk^{-1})^{q/2} \sum_{|x|=s-1}\, \sum_{|y|=t-1} (|F|^q\circ \mD_{l} (\mu_B)) (\bar{x}\oplus \bar{y})
$$
$$
        =
        (qk^{-1})^{q/2} |B|^{-1} \sum_{|y|=t-1} \Cf'_t (|F|^q)^{s-1} (y) \cdot \Cf'_t (B,|F|^q, \dots, |F|^q) (y) \,.
$$
 After that we repeat the argument from 
 \cite{CS}, \cite{sanders2012bogolyubov}, \cite{sanders2013structure} and 
\cite[Theorem  15]{sh_str_survey}. 
    This completes the proof. 
$\hfill\Box$
\end{proof}

\subsection{On a family of  norms}
\label{subs:gen_norms}

In this section we define a very general family of norms, which includes the norms $\E^k_l$, $\mathcal{E}^k_{s,t}$ above, as well as the classical Gowers norms \cite{Gowers_m}. As the reader can see we do not use the Fourier approach in our proofs below.

Let $\Gr$ be an abelian group, $r, k_1,\dots, k_r \ge 2$ be integers and $f:\Gr \to \R$ be an arbitrary function.
Let $K=\prod_{j=1}^r k_j$, $B=[k_1] \times \dots \times [k_r]$ and write $x_1 = (x^{(1)}_1, \dots, x^{(k_1)}_1), \dots, x_r = (x^{(1)}_r, \dots, x^{(k_r)}_r)$. 
Also, for $\o \in B$ we write $\o = (\o_1,\dots,\o_r)$. 
Define 
\[
    \| f\|^K_{E_{k_1,\dots,k_r}} = \sum_{|x_1|=k_1}\, \dots \sum_{|x_r|=k_r}\, \mP_{K} (f) (x_1\oplus \dots \oplus x_r) 
\]
\begin{equation}\label{def:E_gen}
 = \sum_{|x_1|=k_1}\, \dots \sum_{|x_r|=k_r}\, \prod_{\o \in B} f(x^{(\o_1)}_1 + \dots + x^{(\o_r)}_r) 
 \,.
\end{equation}
The case $r=2$ corresponds to $\E^k_l$--norm, $r=3$ is just  $\mathcal{E}^k_{s,t}$--norms and for $k_1= \dots = k_r =2$, we obtain Gowers' $U^k$ norms (up to some normalizations).
In a similar way one can define the multi--scalar product for the quantity $\| \cdot \|_{E_{k_1,\dots,k_r}}$  as was done in \cite{Gowers_m}, namely, having any functions $(f^\o)_{\o \in B}$, we write 
\begin{equation}\label{def:multi_prod}
    \langle f^\o \rangle_{E_{k_1,\dots,k_r}}  = 
\sum_{|x_1|=k_1}\, \dots \sum_{|x_r|=k_r}\, \prod_{\o \in B}            f^{\o} (x^{(\o_1)}_1 + \dots + x^{(\o_r)}_r) \,.
\end{equation}
It is easy to see that if $K$ is an even number, then 
\begin{equation}\label{f:gen_non-negativity}
    \| f\|^K_{E_{k_1,\dots,k_r}} \ge 0 \,.
\end{equation}
    Indeed, let, say, $k_r$ be an even number, then we can write  $x_r$ as $x_r = (x'_r, x''_r)$, where $|x'_r| = |x''_r| = k_r/2$ and whence
\[
    \| f\|^K_{E_{k_1,\dots,k_r}} = 
     \sum_{|x_1|=k_1}\, \dots \sum_{|x_{r-1}|=k_{r-1}}\, \left( \sum_{|x'_r|=k_r/2} \mP_{K/2} (f) (x_1\oplus \dots \oplus x_{r-1} \oplus x'_{r}) \right)^2 \ge 0 \,.
\]
    Also, let us remark the inductive property of the norm $E_{k_1,\dots,k_r}$. For concreteness, we take the $r$th coordinate and  obtain  from definition \eqref{def:E_gen} that 
\begin{equation}\label{f:norm_inductive}
    \| f\|^K_{E_{k_1,\dots,k_r}} = \sum_{|z|=k_r} \| f_z\|^{K/k_r}_{E_{k_1,\dots,k_{r-1}}} \,.
\end{equation}
    Let us make a simple remark concerning $E_{k_1,\dots,k_r}$--norm.  

\begin{lemma}
    Let $f:\Gr \to \R$ be a function. 
    Suppose that there is $j\in [r]$ such that $k_j$ is even and $K/k_j$ is also even. Then $\| f\|_{E_{k_1,\dots,k_r}} = 0$ iff $f\equiv 0$. 
\label{l:f=0}
\end{lemma} 
\begin{proof}
    Without loosing of the generality assume that $j=r$. 
    Write 
\[
    \| f\|^K_{E_{k_1,\dots,k_r}} = 
     \sum_{|x_1|=k_1}\, \dots \sum_{|x_{r-1}|=k_{r-1}}
    \, \left( \sum_{z} \mP_{K/k_r} (f) (x_1\oplus \dots \oplus x_{r-1} \oplus z) \right)^{k_r} = 0 \,.
\]
    Since $k_r$ is an even number, it follows that, in particular, 
    $\sum_z f^{K/k_r} (z) = 0$ (we have taken $x_1=\dots = x_{r-1} =0$ in the last formula) and hence $f\equiv 0$. 
    This completes the proof. 
$\hfill\Box$
\end{proof}

\bp 

Now 
let us show that the multi--scalar product is controlled via $E_{k_1,\dots,k_r}$--norm.

\begin{lemma}
    Let $r\ge 2$ be a positive integer, $k_1,\dots, k_r \ge 2$ be even integers and $f^{\o} : \Gr \to \R$, $\o \in B$ be any functions.  
    Then 
\begin{equation}\label{f:l_gen_norm}
     |\langle f^\o \rangle_{E_{k_1,\dots,k_r}}|
     \le 
     \prod_{\o \in B} \| f^\o\|_{E_{k_1,\dots,k_r}} \,.
\end{equation}
\label{l:gen_norm}
\end{lemma}
\begin{proof} 
    We write 
\[
    \langle f^\o \rangle_{E_{k_1,\dots,k_r}}
    =
    \sum_{|x_1|=k_1}\, \dots \sum_{|x_{r-1}|=k_{r-1}}\, \left( \sum_{x^{(1)}_r}\, \prod_{\o \in B,\, \o_r=1} f^\o (x^{(\o_1)}_1 + \dots + x^{(\o_{r-1})}_{r-1} + x^{(1)}_r )  \right) 
\]
\[
    \dots 
    \left( \sum_{x^{(k_1)}_r}\, \prod_{\o \in B,\, \o_r=k_r} f^\o (x^{(\o_1)}_1 + \dots + x^{(\o_{r-1})}_{r-1} + x^{(k_r)}_r )\right) \,.
\]
    After that 
    apply the H\"older inequality (here we have used the fact that $k_r$ is an even number)
    and we arrive to the new $k_r$ families of functions.
    Take any of them, say, $(\tilde{f}^\o)$, $\o \in B$ and notice that 
\[
    \tilde{f}^\o = \tilde{f}^{\o'} = f^\o  
\]
    for all $\o = (\o_1,\dots,\o_r)$, $\o' = (\o'_1,\dots,\o'_r)$ with $(\o_2,\dots,\o_r) = (\o'_2,\dots,\o'_r)$.
    In particular, the family $(\tilde{f}^\o)$, $\o \in B$ has $K/k_r$ different functions. 
    Now we use the same argument for all remaining variables $x_1,\dots,x_{r-1}$ subsequently changing the families  $(f^\o)$, $\o \in B$. One can easily see that after all these $r$ steps we arrive to $K$ families consisting of single functions $f^\o$, $\o \in B$ 
    (just 
    thanks to 
    the fact that any two points of our box $B$ can be reached by a path in the directions of coordinate  axes). 
    This is equivalent to inequality \eqref{f:l_gen_norm}
    and we 
    complete 
    the proof. 
$\hfill\Box$
\end{proof}

\bp 

Finally, we are ready to 
obtain the main result of this section.
Let us 
write $(k_1,\dots,k_r) \le (m_1,\dots, m_t)$ if the first vector is lexigraphically smaller than the second one (i.e., $r\le t$ and $k_j\le m_j$, $j\in [r]$). 
Also, put
\begin{equation}\label{def:bar-norm}
    \| f\|^K_{\bar{E}_{k_1,\dots,k_r}} = N^{-(k_1 + \dots+k_r)} \| f\|^K_{E_{k_1,\dots,k_r}} \,.
\end{equation} 
    Thus for any $f : \Gr \to [-1,1]$ one has $\| f\|_{\bar{E}_{k_1,\dots,k_r}} \le 1$. 

\begin{theorem}
    Let  $r\ge 2$ be a positive integer, $k_1,\dots, k_r \ge 2$ be even integers and $f : \Gr \to \R$ be a function.  
    Then formula \eqref{def:E_gen} defines a norm of $f$.
    Further if $(k_1,\dots,k_r) \le (m_1,\dots, m_t)$, then 
\begin{equation}\label{f:t_gen_norm}
    \| f\|_{\bar{E}_{k_1,\dots,k_r}}
    \le
    \| f\|_{\bar{E}_{m_1,\dots,m_t}} \,.
\end{equation} 
\label{t:gen_norm}
\end{theorem}
\begin{proof} 
    Take two functions $f,g : \Gr \to \R$. 
    In view of Lemma \ref{l:gen_norm}, we have 
\[
    \| f+g\|^K_{E_{k_1,\dots,k_r}} = 
    \langle f+g \rangle_{E_{k_1,\dots,k_r}}
    \le  \sum_{j=1}^k \binom{K}{j} \| f\|^j_{E_{k_1,\dots,k_r}} \|g \|^{K-j}_{E_{k_1,\dots,k_r}} 
    = 
    ( \| f\|_{E_{k_1,\dots,k_r}} + \| g\|_{E_{k_1,\dots,k_r}} )^K
\]  
    and we have obtained the triangle inequality for $E_{k_1,\dots,k_r}$. 
    By estimate \eqref{f:gen_non-negativity} we know that our quantity $\| f\|_{\bar{E}_{k_1,\dots,k_r}}$ is non--negative. Also, Lemma \ref{l:f=0} guaranties that $\| f\|_{E_{k_1,\dots,k_r}} = 0$ iff $f\equiv 0$. Thus indeed formula \eqref{def:E_gen} defines a norm of $f$.

    It remains to obtain \eqref{f:t_gen_norm}. 
    Let $M=\prod_{j=1}^t m_j$, $B'= [m_1] \times \dots \times [m_t]$, $S = \sum_{j=1}^r k_j$ and  $S' = \sum_{j=1}^t m_j$. 
    Consider the family of functions $(g^\o)_{\o \in B'}$ such that for $\o\in B$ one has $g^\o = f^\o$ and let $g^\o \equiv  1$ otherwise. 
    It is easy to see that 
\[
    \| f\|^K_{\bar{E}_{k_1,\dots,k_r}} N^{S} = \| f\|^K_{E_{k_1,\dots,k_r}} = N^{S-S'} \cdot \langle g^\o \rangle_{E_{m_1,\dots,m_t}}
\]
    Using the last formula, Definition \eqref{def:bar-norm},  as well as Lemma \ref{l:gen_norm}, we obtain 
\[
    \| f\|^K_{\bar{E}_{k_1,\dots,k_r}} N^{S} \le 
    N^{S-S'} \| f\|^K_{E_{m_1,\dots,m_t}} (N^{S'/M})^{M-K}
    = 
    N^{S- S'K/M} \cdot N^{S'K/M} \| f\|^K_{\bar{E}_{m_1,\dots,m_t}}
\]
\[
    = N^S  \| f\|^K_{\bar{E}_{m_1,\dots,m_t}}
\]
as required.
$\hfill\Box$
\end{proof}

\bibliographystyle{abbrv}

\bibliography{bibliography}{}


\end{document}